\documentclass[12pt]{amsart}

\newtheorem{theorem}{Theorem}%[section] (If you want theorem numbered
\newtheorem{proposition}[theorem]{Proposition}

\theoremstyle{definition}

\newtheorem{defn}{Definition}
\theoremstyle{definition}

\theoremstyle{definition}

\author{Indira Chatterji, Ariadna Fossas Tenas and Elise Raphael}
\date{\today}

\title{Under the hood of a carbon footprint calculator}
\begin{document}
\maketitle

\begin{abstract}
We explain the basic linear algebra formulas used in carbon footprints computations.
\end{abstract}
%\section*{Introduction}
\bigskip

Have you computed your carbon footprint lately? Wondered where all those numbers came from? Two main methods for calculating carbon footprints are the {\it Input-Output} method and the {\it Lifecycle Analysis} method, also known as {\it Craddle-to-Grave}. Both aim to quantify the emissions of a product or service, but the methods are fundamentally different.

Input-Output is an economic approach developed in \cite{L}, widely accepted since the 60ies, but only recently adapted to environmental reporting. It uses national or regional data on industries and their inter-dependencies. It relies on aggregated data about economic sectors (e.g., agriculture, manufacturing) and their emissions, and we shall see in this note (Definition \ref{totalintensity}) the formula estimating emissions associated with production, supply chains, and services. This is the commonly accepted method for large-scale, economy-wide assessments, is not data-intensive, relies on publicly available economic data. It will lack detail on specific products and miss the emissions from smaller, niche activities, but offers the warranty that all emissions are taken into account (Proposition~\ref{consistent}).

Lifecycle Analysis estimates a specific product’s environmental impact over its entire lifecycle, from raw material extraction to disposal. However, the under or over counting of secondary emissions are only estimated and the total complexity makes a consistency check impossible to actually carry out.
\subsection*{A concrete question}
To estimate the carbon footprint of a T-shirt for instance, one needs the footprint of the cotton all the way from production of the raw material up to the stitching of the T-shirt, the washing during its lifespan and the disposal of the unusable T-shirt. But to estimate the cotton footprint only, one needs the transport of the water, the production of the pesticides to protect the crop, the fertilizers, all those requiring tractors and fuel, the fuel itself requiring transport and fuel, thus creating loops that are difficult to extricate. And once estimated all the emissions, how do we know that all the emissions of all the objects that we so compute, actually sum up to all the emissions that we actually produce? The Input-Output method takes the problem from the other end: it takes {\it all} the measured emissions, and then attributes it to different sectors, and then computes the carbon footprint of a goods from the textile industry according to its direct emissions, combined with all the indirect emissions as we now explain.
\subsection*{Basic Input-Output}
This method models a closed economy with $n$ {\it economical sectors}, $S=(S_1,\dots,S_n)$, and comes in the following form:
$$\begin{array}{c|cccc|cc}
\ &S_1 & S_2&\dots&S_n &D&T\\
\hline
S_1&c_{11}&c_{12}&\dots&c_{1n}&d_1&t_1\\
S_2&c_{21}&c_{22}&\dots&c_{2n}&d_2&t_2\\
\dots&\dots&\dots&\dots&\dots&\dots&\dots\\
S_n&c_{n1}&c_{n2}&\dots&c_{nn}&d_n&t_n\\
\hline
V&v_1&v_2&\dots&v_n\ &\ \\
T&t_1&t_2&\dots&t_n\ &\ &  
\end{array}$$
The $n\times n$ {\it transaction matrix} is $C=(c_{ij})\in M_n({\bf R}_+)$ (positive coefficients only) has entry $c_{ij}$ the number representing the economical transactions from sector $S_i$ to $S_j$. The diagonal entries $c_{ii}$ then represent the internal gross revenue of the sector. The {\it demand vector} $D=(d_1,\dots,d_n)$ expresses the money spent by the population on each sector: $d_i$ is the amount bought from sector $S_i$, for $i=1,\dots,n$. The {\it added value} vector is given by $V=(v_1,\dots,v_n)$ and the {\it total} vector $T=(t_1,\dots,t_n)$ satisfy the following relations:
$$t_i=\sum_{j=1}^nc_{ij}+d_i=v_i+\sum_{j=1}^nc_{ji}$$
for $i=1,\dots,n$. The value $t_i$ represents the {\it total output} of the sector $S_i$, which is also the {\it total input} because the system is a closed circuit.
\begin{defn}\label{tech}
The {\it technical coefficient matrix} $A=(a_{ij})$ is given by
$$a_{ij}=\frac{c_{ij}}{t_j},$$
so that $A=CD_T^{-1}$ where $D_T$ denotes the diagonal matrix with entries $t_1,\dots,t_n$ in the diagonal.
\end{defn}
Recall that the scalar product of two vectors $X=(xi,\dots,x_n)$ and $Y^t=(y_1,\dots,y_n)$ is given by 
$$\left<X,Y\right>=\sum_{i=1}^nx_iy_i=XY$$
where $X$ is a {\it line vectors}, namely $1\times n$ matrix, and $Y$ is a {\it column vector}, namely an $n\times 1$ matrix and $XY$ is the usual matrix multiplication so a real number.
\subsection*{Footprints}
Economical activities have all sorts of footprints that are correlated to the income and the cost of things. To compute the footprint we need to have a {\it emission} vector $E=(e_1,\dots,e_n)$ where $e_i$ is the direct emission of sector $S_i$. It could be millions of tons of $CO_2$, measured or estimated directly above the industries of the sector, or square feet measured by the occupancy of the industries, or tons of plastic trash estimated by the collectivities, etc. Then $|E|=e_1+\cdots+e_n$ is the total footprint of the economical system. The {\it direct intensity} is given by $F=(f_1,\dots,f_n)$ with
$$f_i=\frac{e_i}{t_i}$$
each entry is the print per unit of money. Now, if we compute the inner product 
$$\left<D, F\right>=\sum_{i=1}^nd_if_i$$ 
it is the footprint directly attributed to the consumers, but it doesn't take into account all the intermediate exchanges that the sectors had to create the product. If you buy goods for $k$ dollars to industry $S_i$, then $kf_i$ indicates the direct footprint of the purchase, but not the total footprint of the goods. That one has to take into account the exchanges the sector $S_i$ had from the other ones. 
\begin{defn}\label{totalintensity}
The {\it total intensity vector} $X=(x_1,\cdots,x_n)$ giving the total footprint of the sector $S_i$ by money amount of final consumer demand, is given by
$$X=F+FA+FA^2+\dots=F(I-A)^{-1}$$
where $I$ denotes the $n\times n$ identity matrix and $A$ the technical coefficients matrix from Definition \ref{tech}.
\end{defn}
Notice that since most matrices are invertible we'll assume that this one is. %
\begin{proposition}\label{consistent}
The total footprint is completely attributed to the consumers. Namely, with the above notations, $\left<X,D\right>=|E|$.
\end{proposition}
\begin{proof}
First notice that the equations $t_i=\sum_{j=1}^nc_{ij}+d_i$ for $i=1,\dots,n$ rephrase as 
$$D=T-C{\bf 1}=T-AD_T{\bf 1}=T-AT=(I-A)T,$$
where ${\bf 1}=(1,\dots,1)$ and $T$ is the total vector and $D$ the demand vector, both viewed as column vectors. We compute:
$$\left<X,D\right>=\left<F(I-A)^{-1},(I-A)T\right>=FT=\sum_{i=1}^n\frac{e_i}{t_i}t_i=|E|$$
\end{proof}
%
%\section{Limitations}
\subsection*{Waste} In a closed economy, in order to take into account waste disposal, one needs a sector $S_i$ that runs the waste management. That sector has a large footprint since incinerators produce large quantities of $CO_2$ (of the order of 1kg $CO_2$ per kg of waste incinerated \cite{Waste}), and is eventually attributed to the consumers through waste management taxes (in France the weight of waste is estimated at half a ton per person and per year). The model also only doesn't discuss demand and waste, assuming that everything is bought and nothing gets wasted.
\subsection*{Corporations} The carbon footprint inventory of a specific firm or corporation is separated in its direct emissions or \emph{scope one}, and its indirect emissions. The latter are separated as \emph{scope two} consisting of the emissions from acquired electricity, heat, and cooling, and the rest is \emph{scope three}, computed using databases of lifecycle analysis \cite{GHGpro}.
\subsection*{Inverses} If $B$ is a matrix $\epsilon$-close to $A$ (say in terms of the coefficients), then the distance between $(I-A)^{-1}$ and $(I-B)^{-1}$ could a priori be as big as we want, for instance if the matrices both have an eigenvalue close enough to 1. Since all the values are estimated, the model is probably not very stable in an optimization exercise.
\subsection*{Added value} The consumer footprint doesn't take into account the added value $V$, which is what the companies make, and the model attributes the footprint solely to the consumers, through the vector $D$. 
We could define {\it the total systemic vector}
$$Y=F(I-A^t)^{-1}$$
and check that $\left<Y,V\right>=|E|$. That vector distributes the total footprint to the added value of the different sectors. That attribution seems more natural than the consumer one, as it links the footprint to the benefits of the companies, see also the 2021 opinion piece \cite{NYT}.
\subsection*{Concretely} The tables are quite available, see for Switzerland \cite{IOCH}, but a measure of environmental impact for each sector appearing in the table is harder to obtain, especially as it needs to be computed/collected in an homogeneous manner over the diﬀerent sectors in order to make sense.
\subsection*{Acknowledgements} This note is based on the article by J. Kitzes \cite{JK}, the talk by J. Steinberger at the Climathics 2022 conference \cite{Clim}, and the third author's talk for the {\it Comission Romande de mathématiques} in 2024.

\end{document}